\documentclass[4pt]{article} 

\usepackage[utf8]{inputenc} 
\usepackage{geometry} 
\usepackage{graphicx} 

\usepackage[colorlinks,citecolor=red]{hyperref}
\usepackage{amsmath}
\usepackage{amsfonts}
\usepackage{dsfont}
\usepackage{mathrsfs}
\usepackage{latexsym}
\usepackage{graphicx}
\usepackage{amssymb}
\usepackage{float}
\usepackage{epsfig}
\usepackage{epstopdf}
\usepackage{color,xcolor}
\usepackage{booktabs} 
\usepackage{array} 
\usepackage{paralist} 
\usepackage{verbatim} 
\usepackage{subfig} 

\newtheorem{theorem}{Theorem}[section]
\newtheorem{lemma}[theorem]{Lemma}
\newtheorem{corollary}[theorem]{Corollary}

\usepackage{fancyhdr} 
\usepackage{sectsty}
\allsectionsfont{\sffamily\mdseries\upshape} 

\usepackage[nottoc,notlof,notlot]{tocbibind} 
\usepackage[titles,subfigure]{tocloft} 

\title{Graph}
\author{Jie$^{a}$\thanks{Corresponding author. Email:~jie\_xue@126.com (J. Xue).}
\\
{\footnotesize $^a$ East China University, Shanghai, China}}
\date{} 

\usepackage{curves}

\def\qed{\hfill$\Box$\vspace{12pt}}

\title{{\bf Graphs with small diameter determined by their $D$-spectra}\thanks{Corresponding author:~Ruifang Liu. E-mail address:~rfliu@zzu.edu.cn(R. Liu). Supported by NSFC (11571323), Outstanding Young Talent Research Fund of Zhengzhou University (No.~1521315002), the China
Postdoctoral Science Foundation (No. 2017M612410) and Foundation for University Key
Teacher of Henan Province (No. 2016GGJS-007).}}
\author{Ruifang Liu$^{a}$~~Jie Xue$^{a}$\\ ~ \\
{\footnotesize $^a$ School of Mathematics and Statistics, Zhengzhou
University, Zhengzhou, Henan 450001, China}\\}
\date{}

\begin{document}
\maketitle

\begin{abstract}
Let $G$ be a connected graph with vertex set
$V(G)=\{v_{1},v_{2},\ldots,v_{n}\}$. The distance matrix
$D(G)=(d_{ij})_{n\times n}$ is the matrix indexed by the vertices of
$G,$ where $d_{ij}$ denotes the distance between the vertices
$v_{i}$ and $v_{j}$. Suppose that
$\lambda_{1}(D)\geq\lambda_{2}(D)\geq\cdots\geq\lambda_{n}(D)$ are
the distance spectrum of $G$. The graph $G$ is said to be determined
by its $D$-spectrum if with respect to the distance matrix $D(G)$,
any graph having the same spectrum as $G$ is isomorphic to $G$. In
this paper, we give the distance characteristic polynomial of some
graphs with small diameter, and also prove that these graphs are
determined by their $D$-spectra.

\bigskip
\noindent {\bf AMS Classification:} 05C50

\noindent {\bf Key words:} Distance spectrum; Distance
characteristic polynomial; $D$-spectrum determined
\end{abstract}

\section{Introduction}

~~~~All graphs considered here are simple, undirected and connected.
Let $G$ be a graph with vertex set
$V(G)=\{v_{1},v_{2},\ldots,v_{n}\}$ and edge set $E(G)$. Two
vertices $u$ and $v$ are called adjacent if they are connected by an
edge, denoted by $u\sim v$. Let $N_{G}(v)$ denote the neighbor set
of $v$ in $G$. The degree of a vertex $v,$ written by $d_{G}(v)$ or
$d(v)$, is the number of edges incident with $v$. Let $X$ and $Y$ be
subsets of vertices of $G$. The induced subgraph $G[X]$ is the
subgraph of $G$ whose vertex set is $X$ and whose edge set consists
of all edges of $G$ which have both ends in $X$. We denote by
$E[X,Y]$ the set of edges of $G$ with one end in $X$ and the other
end in $Y$, and denote by $e[X,Y]$ their number. The distance
between vertices $u$ and $v$ of a graph $G$ is denoted by
$d_G(u,v)$. The diameter of $G,$ denoted by $\mbox {diam}(G),$ is
the maximum distance between any pair of vertices of $G.$ The
complete product $G_{1}\bigtriangledown G_{2}$ of graphs $G_{1}$ and
$G_{2}$ is the graph obtained from $G_{1}\cup G_{2}$ by joining
every vertex of $G_{1}$ to every vertex of $G_{2}.$ Denote by
$K_{n}$, $C_{n}$, $P_{n}$ and $S_{n}$ the complete graph, the cycle,
the path and the star, respectively, each on $n$ vertices. Let
$K_{n}^{c}$ denote the complement of $K_{n}.$

The distance matrix $D(G)=(d_{ij})_{n\times n}$ of a connected graph
$G$ is the matrix indexed by the vertices of $G,$ where $d_{ij}$
denotes the distance between the vertices $v_{i}$ and $v_{j}$. Let
$\lambda_{1}(D)\geq\lambda_{2}(D)\geq\cdots\geq\lambda_{n}(D)$ be
the spectrum of $D(G)$, that is, the distance spectrum of $G.$ The
polynomial $P_{D}(\lambda)=\mbox{det}(\lambda I-D(G))$ is defined as
the distance characteristic polynomial of a graph $G.$ Two graphs
are said to be $D$-cospectral if they have the same distance
spectrum. A graph $G$ is said to be determined by its $D$-spectrum
if there is no other non-isomorphic graph $D$-cospectral to $G$.

Which graphs are determined by their spectrum seems to be a
difficult and interesting problem in spectral graph theory.
This question was raised by G\"{u}nthard and Primas \cite{GP}.
For surveys of this question see \cite{ER1, ER3}. Up to now, only a few families of graphs were shown to be determined
by their spectra, most of which were restricted to the adjacency,
Laplacian or signless Laplacian spectra. In particular, there are
much fewer results on which graphs are determined by their
$D$-spectra. In \cite{LHQ}, Lin et al. proved that the complete
graph $K_{n},$ the complete bipartite graph $K_{n_{1},n_{2}}$ and
the complete split graph $K_{a}\bigtriangledown K_{b}^{c}$ are
determined by their $D$-spectra, and the authors proposed a conjecture
that the complete $k$-partite graph $K_{n_{1}, n_{2}, \ldots,
n_{k}}$ is determined by its $D$-spectrum. Recently, Jin and Zhang
\cite{XDZ} have confirmed the conjecture. Lin, Zhai and Gong
\cite{LZG} characterized all connected graphs with
$\lambda_{n-1}(D(G))=-1$, and showed that these graphs are
determined by their $D$-spectra. Moreover, in this paper, they also
proved that the graphs with $\lambda_{n-2}(D(G))>-1$ are determined
by their distance spectra. In \cite{LHQ1}, Lin showed that connected
graphs with $\lambda_{n}(D(G))\geq-1-\sqrt{2}$ are determined by
their distance spectra. Cioab\u{a} et al. \cite{CI} affirmed that the famous
friendship graph $F_{n}^{k}~(k\neq16)$ is determined by its adjacency spectrum.
Lu, Huang and Huang \cite{LL} show that all graphs with exactly two distance eigenvalues (counting multiplicity)
different from -1 and -3 are determined by their $D$-spectra, and particularly, $F_{n}^{k}$ is determined by its distance spectrum.

Next, we introduce a class of graphs $K^{n_{1},n_{2},\ldots, n_{k}}_{n}$, as shown in Fig. 1.

\begin{center}
\setlength{\unitlength}{1.0mm}
\begin{picture}(420,10)

\put(80,0){\circle*{1}}
\put(65,13){\bigcircle{10}}\put(62,12){$K_{n_{2}}$}
\put(65,-13){\bigcircle{10}}\put(62,-14){$K_{n_{1}}$}
\put(95,13){\bigcircle{10}}\put(92,12){$K_{n_{3}}$}
\put(95,-13){\bigcircle{10}}\put(92,-14){$K_{n_{k}}$}

\curve(69,16,80,0)\curve(91,16,80,0)\curve(69,-16,80,0)\curve(91,-16,80,0)
\curve(63,8.4,80,0)\curve(97,8.4,80,0)\curve(63,-8.4,80,0)\curve(97,-8.4,80,0)
\put(95,0){\circle*{0.55}}\put(95,3){\circle*{0.55}}\put(95,-3){\circle*{0.55}}
\put(79,-5){$v$}
\end{picture}
\vskip 2.0cm  Fig. $1$. Graph $K^{n_{1},n_{2},\ldots, n_{k}}_{n}$.
\end{center}
$\bullet$ $K^{n_{1},n_{2},\ldots, n_{k}}_{n}$: $=(K_{n_{1}}\cup K_{n_{2}}\cup\cdots \cup K_{n_{k}})\bigtriangledown\{v\},$ where $k\geq2$.

In this paper, we firstly show that three special classes of graphs in $K^{n_{1},n_{2},\ldots, n_{k}}_{n},$ that is, $K^{h}_{n}=K^{h-1,1,\ldots,1}_{n}~(4\leq h \leq n-1),$ $K^{s,t}_{n}~(s\geq4$ and $t\geq4)$ and $K^{n_{1},n_{2},\ldots, n_{k}}_{n}~(1\leq n_{i}\leq 2)$ are
determined by their $D$-spectra. Clearly, the friendship graph $F_{n}^{k}$ belongs to the third class.

Secondly, we prove that $K^{s+t}_{n}~(s\geq2, t\geq2)$ (see Fig. 2) is also determined by its $D$-spectrum.

\begin{center}
\setlength{\unitlength}{1.0mm}
\begin{picture}(105,10)
\put(-4,5){\circle*{1}} \put(-4,5){\line(3,4){6}}
\put(-4,5){\line(3,2){6}} \put(-4,5){\line(3,-4){6}}
\put(-4,5){\line(3,-2){6}} \put(2,13){\circle*{1}}
\put(2,9){\circle*{1}} \put(2,1){\circle*{1}}
\put(2,-3){\circle*{1}} \put(1.5,3.5){\vdots}
\put(-11,5){\circle{14}} \put(-13,4){$K_{h}$}
\put(-8,-12){$K^{h}_{n}$} \put(26,5){\circle*{1}}
\put(33,5){\circle*{1}} \put(26,5){\line(1,0){7}}
\put(19,5){\circle{14}} \put(40,5){\circle{14}} \put(17,4){$K_{s}$}
\put(38,4){$K_{t}$} \put(27,-12){$K^{s+t}_{n}$}
\put(63,5){\circle{14}} \put(77,5){\circle{14}}
\put(70,5){\circle*{1}} \put(61,4){$K_{s}$} \put(75,4){$K_{t}$}
\put(67,-12){$K_{n}^{s,t}$}

\put(110,0){\line(-5,1){15}}\put(110,0){\line(-5,-1){15}}
\put(110,0){\line(5,1){15}}\put(110,0){\line(5,-1){15}}
\put(110,0){\line(-1,5){3}}\put(110,0){\line(1,5){3}}
\put(110,0){\line(-3,2){12}}\put(110,0){\line(-2,3){8}}
\put(110,0){\line(3,2){12}} \put(110,0){\line(2,3){8}}

\put(95,-3){\circle*{1}}\put(95,3){\circle*{1}}\put(95,3){\line(0,-1){6}}
\put(107,15){\circle*{1}}\put(113,15){\circle*{1}}\put(107,15){\line(1,0){6}}
\put(98,8){\circle*{1}}\put(102,12){\circle*{1}}\put(98,8){\line(1,1){4}}

\put(122,8){\circle*{1}}\put(118,12){\circle*{1}}\put(122,8){\line(-1,1){4}}

\put(125,-3){\circle*{1}}\put(125,3){\circle*{1}}\put(125,-3){\line(0,1){6}}
\put(107,-4){\circle*{0.55}}\put(110,-4){\circle*{0.55}}\put(113,-4){\circle*{0.55}}
\put(108,-12){$F_{n}^{k}$}

\end{picture}
\vskip 1.5cm  Fig. $2$. Graphs $K^{h}_{n}$, $K^{s+t}_{n}$,
$K_{n}^{s,t}$ and $F_{n}^{k}$.
\end{center}
$\bullet$ $K^{s+t}_{n}$: the graph obtained by adding one edge joining a vertex of $K_{s}$ to a vertex of $K_{t}$.

\section{Preliminaries}

~~~~In this section, we give some useful lemmas and results. The
following lemma is well-known Cauchy interlacing theorem.

\begin{lemma}{\bf (\cite{DMC})} \label{le1}
Let $A$ be a Hermitian matrix of order $n$ with eigenvalues
$\lambda_{1}(A)\geq\lambda_{2}(A)\geq\cdots \geq\lambda_{n}(A),$ and
$B$ be a principal submatrix of $A$ of order $m$ with eigenvalues
$\mu_{1}(B)\geq \mu_{2}(B)\geq\cdots \geq\mu_{m}(B)$. Then
$\lambda_{n-m+i}(A)\leq\mu_{i}(B)\leq\lambda_{i}(A)$ for
$i=1,2,\ldots,m.$
\end{lemma}

Applying Lemma \ref{le1} to the distance matrix $D$ of a graph, we
have

\begin{lemma}\label{le15}
Let $G$ be a graph of order $n$ with distance spectrum
$\lambda_{1}(G)\geq\lambda_{2}(G)\geq\cdots \geq\lambda_{n}(G),$ and
$H$ be an induced subgraph of $G$ on $m$ vertices with the distance
spectrum $\mu_{1}(H)\geq \mu_{2}(H)\geq\cdots \geq\mu_{m}(H).$
Moreover, if $D(H)$ is a principal submatrix of $D(G),$ then
$\lambda_{n-m+i}(G)\leq\mu_{i}(H)\leq\lambda_{i}(G)$ for
$i=1,2,\ldots,m.$
\end{lemma}

\begin{lemma}{\bf (\cite{LHQ})} \label{le2}
Let $G$ be a connected graph and $D$ be the distance matrix of G.
Then $\lambda_{n}(D)=-2$ with multiplicity $n-k$ if and only if $G$
is a complete $k$-partite graph for $2\leq k\leq n-1$.
\end{lemma}

\begin{lemma}{\bf (\cite{XJ})}\label{le18}
Let $G$ be a graph with order $n$ and $d(G)=2$. If $G^{'}$ has the same distance spectrum as $G$, then\par
\noindent$\bullet |E(G)|=|E(G^{'})|$ when $d(G^{'})=2;$\par
\noindent$\bullet |E(G)|<|E(G^{'})|$ when $d(G^{'})\geq3$.
\end{lemma}

\begin{theorem}\label{th1}
Let $4\leq h\leq n-1.$ The distance characteristic polynomial of
$K_{n}^{h}$ is
$$P_{D}(\lambda)=(\lambda+1)^{h-2}(\lambda+2)^{n-h-1}[\lambda^{3}+(h+4-2n)\lambda^{2}+(5-2h-2nh+2h^{2}-n)\lambda-nh+h^{2}-2h+2].$$
Let $\lambda_{1}\geq \lambda_{2}\geq \cdots\geq\lambda_{n}$ be the
distance spectrum of $K_{n}^{h}$. Then \\
$\bullet$ $\lambda_{1}>0$, $-1<\lambda_{2}<-\frac{1}{2}$ and $\lambda_{3}=-1.$\\
$\bullet$ $\lambda_{n-1}\in\{-1,-2\}$ and $\lambda_{n}<-2.$
\end{theorem}

\begin{proof}It is clear that the diameter of $K_{n}^{h}$ is 2, and the distance matrix of $K_{n}^{h}$ is
\begin{equation*}\begin{split}
D&=\left(\begin{array}{ccccccc}
0&\cdots&1&1&2&\cdots&2\\
\vdots&\ddots&\vdots&\vdots&\vdots&\vdots&\vdots\\
1&\cdots&0&1&2&\cdots&2\\
1&\cdots&1&0&1&\cdots&1\\
2&\cdots&2&1&0&\cdots&2\\
\vdots&\vdots&\vdots&\vdots&\vdots&\ddots&\vdots\\
2&\cdots&2&1&2&\cdots&0\\
\end{array}\right).\\
\end{split}\end{equation*}
Then
\begin{equation*}\begin{split}
&\mbox{det}(\lambda I-D)=\left|\begin{array}{ccccccc}
\lambda&\cdots&-1&-1&-2&\cdots&-2\\
\vdots&\ddots&\vdots&\vdots&\vdots&\vdots&\vdots\\
-1&\cdots&\lambda&-1&-2&\cdots&-2\\
-1&\cdots&-1&\lambda&-1&\cdots&-1\\
-2&\cdots&-2&-1&\lambda&\cdots&-2\\
\vdots&\vdots&\vdots&\vdots&\vdots&\ddots&\vdots\\
-2&\cdots&-2&-1&-2&\cdots&\lambda\\
\end{array}\right|\\
&=\left|\begin{array}{cccccccccc}
\lambda-(h-2)&-1&\cdots&-1&-1&-2-2(n-h-1)&-2&\cdots&-2\\
0&\lambda+1&\cdots&0&0&0&0&\cdots&0\\
\vdots&\vdots&\ddots&\vdots&\vdots&\vdots&\vdots&\vdots&\vdots\\
0&0&\cdots&\lambda+1&0&0&0&\cdots&0\\
-1-(h-2)&-1&\cdots&-1&\lambda&-1-(n-h-1)&-1&\cdots&-1\\
-2-2(h-2)&-2&\cdots&-2&-1&\lambda-2(n-h-1)&-2&\cdots&-2\\
0&0&\cdots&0&0&0&\lambda+2&\cdots&0\\
\vdots&\vdots&\vdots&\vdots&\vdots&\vdots&\vdots&\ddots&\vdots\\
0&0&\cdots&0&0&0&0&\cdots&\lambda+2\\
\end{array}\right|\\
\end{split}\end{equation*}
\begin{equation*}\begin{split}
&=(\lambda+1)^{h-2}(\lambda+2)^{n-h-1}\left|\begin{array}{cccccccccc}
\lambda-(h-2)&-1&-2-2(n-h-1)\\
-1-(h-2)&\lambda&-1-(n-h-1)\\
-2-2(h-2)&-1&\lambda-2(n-h-1)\\
\end{array}\right|\\
&=(\lambda+1)^{h-2}(\lambda+2)^{n-h-1}[\lambda^{3}+(h+4-2n)\lambda^{2}+(5-2h-2nh+2h^{2}-n)\lambda-nh+h^{2}-2h+2].
\end{split}\end{equation*}

In the following, we will prove the remaining part of Theorem
\ref{th1}. Consider the cubic function on $x$
$$f(x)=x^{3}+(h+4-2n)x^{2}+(5-2h-2nh+2h^{2}-n)x-nh+h^{2}-2h+2.$$ From a
simple calculation, we have
\begin{equation*}\begin{split}
\begin{cases}
f(0)=-nh+h^{2}-2h+2=-h(n-h)-(2h-2)<0,\\
f(-\frac{1}{2})=\frac{3}{8}-\frac{3}{4}h<0,\\
f(-1)=h-n+nh-h^{2}=(n-h)(h-1)>0,\\
f(-2)=6h-6n+3nh-3h^{2}=(n-h)(3h-6)>0.
\end{cases}
\end{split}\end{equation*}
Note that $f(x)\rightarrow +\infty ~(x\rightarrow +\infty)$ and
$f(0)<0$, so there is at least one root in $(0,+\infty)$. Since
$f(-\frac{1}{2})<0$ and $f(-1)>0$, then there is at least one root
in $(-1,-\frac{1}{2})$. By $f(x)\rightarrow -\infty ~(x\rightarrow
-\infty)$ and $f(-2)>0$, so there is at least one root in
$(-\infty,-2)$. Thus there is exactly one root in each interval.
\qed
\end{proof}

Using the similar method to compute the distance characteristic
polynomials of $K^{s+t}_{n}$ and $K^{s,t}_{n},$ we have the
following two results.

\begin{theorem}\label{th2}
Let $s\geq2, t\geq2$ and $n=s+t.$ Then the distance characteristic
polynomial of $K^{s+t}_{n}$ is
$$P_{D}(\lambda)=(\lambda+1)^{n-4}[\lambda^{4}+(-s-t+4)\lambda^{3}+(2t+2s-8st+4)\lambda^{2}+(6s+6t-14st)\lambda-5st+2s+2t].$$
Let $\lambda_{1}\geq \lambda_{2}\geq \cdots\geq\lambda_{n}$ denote
the distance spectrum of $K^{s+t}_{n}.$ Then \\
$\bullet$ $\lambda_{1}>0$, $-1<\lambda_{2}<-\frac{1}{2}$ and $\lambda_{3}=-1.$\\
$\bullet$ $-2<\lambda_{n-1}<-1$ and $\lambda_{n}<-2.$
\end{theorem}

\begin{proof}
The distance matrix of $K^{s+t}_{n}$ is
\begin{equation*}\begin{split}
D&=\left(\begin{array}{cccccccc}
0&\cdots&1&1&2&3&\cdots&3\\
\vdots&\ddots&\vdots&\vdots&\vdots&\vdots&\vdots&\vdots\\
1&\cdots&0&1&2&3&\cdots&3\\
1&\cdots&1&0&1&2&\cdots&2\\
2&\cdots&2&1&0&1&\cdots&1\\
3&\cdots&3&2&1&0&\cdots&1\\
\vdots&\vdots&\vdots&\vdots&\vdots&\vdots&\ddots&\vdots\\
3&\cdots&3&2&1&1&\cdots&0\\
\end{array}\right).\\
\end{split}\end{equation*}
Similar to the proof of Theorem \ref{th1}, by a simple calculation,
we have
\begin{equation*}\begin{split}
&\mbox{det}(\lambda
I-D)=(\lambda+1)^{n-4}\left|\begin{array}{cccccccccc}
\lambda-(s-2)&-1&-2&-3-3(t-2)\\
-1-(s-2)&\lambda&-1&-2-2(t-2)\\
-2-2(s-2)&-1&\lambda&-1-(t-2)\\
-3-3(s-2)&-2&-1&\lambda-(t-2)\\
\end{array}\right|\\
&=(\lambda+1)^{n-4}[\lambda^{4}+(-s-t+4)\lambda^{3}+(2t+2s-8st+4)\lambda^{2}+(6s+6t-14st)\lambda-5st+2s+2t].
\end{split}\end{equation*}
Consider the quartic function on $x$
$$f(x)=x^{4}+(-s-t+4)x^{3}+(2t+2s-8st+4)x^{2}+(6s+6t-14st)x-5st+2s+2t.$$
Note that $(s-1)(t-1)=st-s-t+1>0$, hence $st+1>s+t$. Then we obtain
that
\begin{equation*}\begin{split}
\begin{cases}
f(0)=-5st+2s+2t<2(st+1)-5st=2-3st<0,\\
f(-\frac{1}{2})=\frac{9}{16}-\frac{3}{8}s-\frac{3}{8}t<0,\\
f(-1)=1-s-t+st>0,\\
f(-2)=6s+6t-9st<6(st+1)-9st=6-3st<0.
\end{cases}
\end{split}\end{equation*}
Note that $f(x)\rightarrow +\infty~(x\rightarrow +\infty)$ and
$f(0)<0$, so there is at least one root in $(0, +\infty)$. Since
$f(-\frac{1}{2})<0$ and $f(-1)>0$, then there is at least one root
in $(-1,-\frac{1}{2})$. Since $f(-1)>0$ and $f(-2)<0$, then there is
at least one root in $(-2,-1)$. By $f(x)\rightarrow
+\infty~(x\rightarrow -\infty)$ and $f(-2)<0$, so there is at least
one root in $(-\infty,-2)$. Thus there is exactly one root in each
interval. The result is completed. \qed
\end{proof}

\begin{theorem}\label{th3}
Let $s\geq4, t\geq4$ and $n=s+t-1.$ Then the distance characteristic
polynomial of $K^{s,t}_{n}$ is
$$P_{D}(\lambda)=(\lambda+1)^{n-3}[\lambda^{3}+(-s-t+4)\lambda^{2}+(2+s+t-3st)\lambda+s+t-2st].$$
Let $\lambda_{1}\geq \lambda_{2}\geq \cdots\geq\lambda_{n}$ denote
the distance spectrum of $K^{s,t}_{n}.$ Then \\
$\bullet$ $\lambda_{1}>0$, $-1<\lambda_{2}<-\frac{2}{3}$ and
$\lambda_{3}=-1.$ \\
$\bullet$ $\lambda_{n-1}=-1$ and $\lambda_{n}<-2.$
\end{theorem}

\begin{proof}
The distance matrix of $K^{s,t}_{n}$ is
\begin{equation*}\begin{split}
D&=\left(\begin{array}{cccccccc}
0&\cdots&1&1&2&\cdots&2\\
\vdots&\ddots&\vdots&\vdots&\vdots&\vdots&\vdots\\
1&\cdots&0&1&2&\cdots&2\\
1&\cdots&1&0&1&\cdots&1\\
2&\cdots&2&1&0&\cdots&1\\
\vdots&\vdots&\vdots&\vdots&\vdots&\ddots&\vdots\\
2&\cdots&2&1&1&\cdots&0\\
\end{array}\right).\\
\end{split}\end{equation*}
Similar to the proof of Theorem \ref{th1}, we have
\begin{equation*}\begin{split}
&\mbox{det}(\lambda
I-D)=(\lambda+1)^{n-3}\left|\begin{array}{cccccccccc}
\lambda-(s-2)&-1&-2-2(t-2)\\
-1-(s-2)&\lambda&-1-(t-2)\\
-2-2(s-2)&-1&\lambda-(t-2)\\
\end{array}\right|\\
&=(\lambda+1)^{n-3}[\lambda^{3}+(-s-t+4)\lambda^{2}+(2+s+t-3st)\lambda+s+t-2st].
\end{split}\end{equation*}
Consider the cubic function on $x$
$$f(x)=x^{3}+(-s-t+4)x^{2}+(2+s+t-3st)x+s+t-2st.$$ Note that
$(s-1)(t-1)=st-s-t+1>0$, then $st+1>s+t$. By a simple calculation,
we have
\begin{equation*}\begin{split}
\begin{cases}
f(0)=s+t-2st<1-st<0,\\
f(-\frac{2}{3})=\frac{4}{27}-\frac{1}{9}s-\frac{1}{9}t<0,\\
f(-1)=1-s-t+st>0.
\end{cases}
\end{split}\end{equation*}
Note that $f(x)\rightarrow +\infty~(x\rightarrow +\infty)$ and
$f(0)<0$, so there is at least one root in $(0,+\infty)$. Since
$f(-\frac{2}{3})<0$ and $f(-1)>0$, then there is at least one root
in $(-1,-\frac{2}{3})$. Since $f(-1)>0$ and $f(x)\rightarrow
-\infty~(x\rightarrow -\infty)$, then there is at least one root in
$(-\infty,-1)$. Thus there is exactly one root in each interval.
This means that $\lambda_{1}>0$, $-1<\lambda_{2}<-\frac{2}{3}$,
$\lambda_{3}=\lambda_{n-1}=-1$ and $\lambda_{n}<-1$.

Obviously, the diameter of $K^{s,t}_{n}$ is 2, and $P_{3}$ is an
induced subgraph of $K^{s,t}_{n}$. Moreover, $D(P_{3})$ is a
principal submatrix of $D(K^{s,t}_{n}).$ It is easy to calculate
that $\lambda_{3}(P_{3})=-2,$ then by Lemma \ref{le15},
$\lambda_{n}(K^{s,t}_{n})\leq \lambda_{3}(P_{3})=-2.$ Furthermore,
$K^{s,t}_{n}$ is not a complete $k$-partite graph, then by Lemma
\ref{le2}, we have $\lambda_{n}<-2.$ \qed
\end{proof}

By Theorems \ref{th1}, \ref{th2} and \ref{th3}, we obtain the
following corollary.

\begin{corollary}\label{co1}
No two non-isomorphic graphs of $K^{h}_{n},$ $K^{s+t}_{n}$ and
$K^{s,t}_{n}$ are $D$-cospectral.
\end{corollary}

\begin{proof}
From the distance characteristic polynomials of $K^{h}_{n},$
$K^{s+t}_{n}$ and $K^{s,t}_{n}$, for any two non-isomorphic graphs
belonging to the same type, the result is obvious.

It is clear that $K^{s+t}_{n}$ and $K^{s,t}_{n}$ have distinct
distance spectrum, since -1 is the distance eigenvalue of
$K^{s+t}_{n}$ with multiplicity $n-4$, and is the distance
eigenvalue of $K^{s,t}_{n}$ with multiplicity $n-3,$ respectively.

Now we only need to prove that $K^{h}_{n}$ has distinct distance
spectrum with $K^{s+t}_{n}$ and $K^{s,t}_{n}$.

Suppose that $K^{h}_{n}$ and $K^{s+t}_{n}$ are $D$-cospectral. Note
that -1 is the distance eigenvalue of $K^{s+t}_{n}$ with
multiplicity $n-4,$ then -1 is also the distance eigenvalue of
$K^{h}_{n}$ with multiplicity $n-4.$ On the other hand, notice that
-2 is not the distance eigenvalue of $K^{s+t}_{n}$, then it follows
that -2 is also not the distance eigenvalues of $K^{h}_{n}$, thus
$n=h+1$. Then -1 is the distance eigenvalue of $K^{h}_{n}$ with
multiplicity $n-3$, a contradiction.

Assume that $K^{h}_{n}$ and $K^{s,t}_{n}$ are $D$-cospectral. Note
that -2 is not the distance eigenvalue of $K^{s,t}_{n},$ then it
follows that -2 is also not the distance eigenvalue of $K^{h}_{n}$,
so $n=h+1$. Then we have
\begin{equation*}\begin{split}
\begin{cases}
P_{D(K^{h}_{n})}(\lambda)=(\lambda+1)^{n-3}[\lambda^{3}+(-n+3)\lambda^{2}+(-5n+9)\lambda-3n+5],\\
P_{D(K^{s,t}_{n})}(\lambda)=(\lambda+1)^{n-3}[\lambda^{3}+(-s-t+4)\lambda^{2}+(2+s+t-3st)\lambda+s+t-2st].
\end{cases}
\end{split}\end{equation*}
Note that they have the same distance characteristic polynomial,
hence
\begin{equation*}\begin{split}
\begin{cases}
-3n+5=s+t-2st,\\
n=s+t-1.
\end{cases}
\end{split}\end{equation*}
Solving the two equations we get $t=2$ or $t=n-1$, a contradiction.\qed
\end{proof}

\section{Main results}

~~~~In this section, our first task is to show that $K^{h}_{n},$
$K^{s+t}_{n}$ and $K^{s,t}_{n}$ are determined by their $D$-spectra.
First, we give some useful graphs and their distance spectra.

\vskip 0.5cm
\begin{center}
\setlength{\unitlength}{1.0mm}
\begin{picture}(300,10)

\put(3,8){$P_{4}$} \put(10,9){\line(1,0){18}}
\put(10,9){\circle*{1}} \put(16,9){\circle*{1}}
\put(22,9){\circle*{1}} \put(28,9){\circle*{1}}

\put(3,2){$P_{5}$} \put(10,3){\line(1,0){24}}
\put(10,3){\circle*{1}} \put(16,3){\circle*{1}}
\put(22,3){\circle*{1}} \put(28,3){\circle*{1}}
\put(34,3){\circle*{1}}

\put(45,-3){$C_{4}$} \put(44,2){\line(1,0){8}}
\put(44,10){\line(1,0){8}} \put(52,2){\line(0,1){8}}
\put(44,2){\line(0,1){8}} \put(44,2){\circle*{1}}
\put(44,10){\circle*{1}} \put(52,2){\circle*{1}}
\put(52,10){\circle*{1}} \put(68,-3){$C_{5}$}
\put(62,2){\line(1,0){16}} \put(62,2){\circle*{1}}
\put(70,2){\circle*{1}} \put(78,2){\circle*{1}}
\put(62,2){\line(0,1){8}} \put(62,10){\line(1,0){8}}
\put(78,2){\line(-1,1){8}} \put(70,10){\circle*{1}}
\put(62,10){\circle*{1}} \put(98,-3){$H_{1}$}
\put(88,2){\line(1,0){24}} \put(104,10){\circle*{1}}
\put(88,2){\circle*{1}} \put(96,2){\circle*{1}}
\put(104,2){\circle*{1}} \put(112,2){\circle*{1}}
\put(96,2){\line(1,1){8}} \put(104,2){\line(0,1){8}}
\put(88,2){\line(2,1){16}} \put(112,2){\line(-1,1){8}}
\put(132,-3){$H_{2}$} \put(122,2){\line(1,0){24}}
\put(130,10){\circle*{1}} \put(122,2){\circle*{1}}
\put(130,2){\circle*{1}} \put(138,2){\circle*{1}}
\put(146,2){\circle*{1}} \put(130,2){\line(0,1){8}}
\put(122,2){\line(1,1){8}} \put(138,2){\line(-1,1){8}}

\put(20,-20){$H_{3}$}
\put(10,-15){\line(1,0){24}} \put(10,-15){\line(1,0){16}} \put(10,-15){\circle*{1}} \put(18,-15){\circle*{1}} \put(26,-15){\circle*{1}} \put(34,-15){\circle*{1}} \put(18,-7){\circle*{1}} \put(18,-15){\line(0,1){8}} \put(26,-15){\line(-1,1){8}}

\put(49,-20){$H_{4}$}
\put(44,-15){\line(1,0){16}} \put(44,-15){\circle*{1}} \put(52,-15){\circle*{1}} \put(60,-15){\circle*{1}} \put(44,-7){\circle*{1}} \put(52,-7){\circle*{1}}
\put(44,-15){\line(0,1){8}} \put(52,-15){\line(0,1){8}} \put(52,-15){\line(-1,1){8}} \put(44,-7){\line(1,0){8}}
\put(110,-20){$H_{6}$}
\put(102,-11){\line(1,0){18}} \put(102,-11){\circle*{1}} \put(108,-11){\circle*{1}} \put(114,-11){\circle*{1}} \put(120,-11){\circle*{1}} \put(104,-15){\circle*{1}} \put(104,-7){\circle*{1}} \put(108,-11){\line(-1,1){4}} \put(108,-11){\line(-1,-1){4}}

\put(80,-20){$H_{5}$}
\put(70,-15){\line(1,0){24}} \put(70,-15){\circle*{1}} \put(78,-15){\circle*{1}}  \put(86,-15){\circle*{1}} \put(94,-15){\circle*{1}}
\put(78,-15){\line(0,1){8}} \put(78,-15){\line(-1,1){8}} \put(78,-7){\circle*{1}} \put(70,-7){\circle*{1}} \put(70,-15){\line(0,1){8}}

\put(135,-20){$H_{7}$}
\put(128,-15){\line(1,0){18}} \put(128,-15){\circle*{1}} \put(134,-15){\circle*{1}} \put(140,-15){\circle*{1}} \put(146,-15){\circle*{1}} \put(134,-9){\circle*{1}} \put(140,-9){\circle*{1}} \put(134,-15){\line(0,1){6}} \put(140,-15){\line(0,1){6}}

\put(19,-37){$H_{8}$}
\put(10,-32){\line(1,0){21}} \put(10,-32){\circle*{1}} \put(17,-32){\circle*{1}} \put(24,-32){\circle*{1}} \put(31,-32){\circle*{1}} \put(10,-25){\circle*{1}} \put(24,-25){\circle*{1}} \put(10,-32){\line(0,1){7}} \put(17,-32){\line(-1,1){7}} \put(24,-32){\line(0,1){7}}

\put(46,-37){$H_{9}$}
\put(40,-28){\line(1,0){16}} \put(40,-28){\line(1,1){4}} \put(40,-28){\line(1,-1){4}} \put(44,-32){\line(0,1){8}} \put(48,-28){\line(-1,1){4}} \put(48,-28){\line(-1,-1){4}} \put(56,-28){\line(-3,1){12}} \put(56,-28){\line(-3,-1){12}}
\put(40,-28){\circle*{1}} \put(48,-28){\circle*{1}} \put(56,-28){\circle*{1}} \put(44,-24){\circle*{1}} \put(44,-32){\circle*{1}}

\put(72,-37){$H_{10}$}
\put(66,-32){\line(1,0){16}} \put(66,-32){\line(1,1){8}} \put(74,-32){\line(0,1){8}} \put(74,-32){\line(1,1){8}} \put(82,-32){\line(-1,1){8}} \put(74,-24){\line(1,0){8}} \put(66,-32){\circle*{1}} \put(74,-32){\circle*{1}} \put(82,-32){\circle*{1}} \put(74,-24){\circle*{1}} \put(82,-24){\circle*{1}}

\put(98,-37){$H_{12}$}
\put(92,-32){\line(1,0){16}} \put(92,-32){\line(1,1){8}} \put(100,-32){\line(0,1){8}} \put(100,-32){\line(1,1){8}} \put(108,-32){\line(-1,1){8}} \put(100,-24){\line(1,0){8}} \put(92,-32){\circle*{1}} \put(100,-32){\circle*{1}} \put(108,-32){\circle*{1}} \put(100,-24){\circle*{1}} \put(108,-24){\circle*{1}} \put(100,-32){\line(-1,1){8}} \put(100,-24){\line(-1,0){8}} \put(92,-32){\line(0,1){8}} \put(108,-32){\line(0,1){8}} \put(92,-24){\circle*{1}}
\put(16,-54){$H_{13}$}
\put(10,-49){\line(1,0){16}} \put(10,-49){\line(0,1){8}} \put(26,-49){\line(0,1){8}} \put(18,-49){\line(0,1){8}} \put(18,-49){\line(1,1){8}} \put(18,-49){\line(-1,1){8}} \put(10,-49){\circle*{1}} \put(18,-49){\circle*{1}} \put(18,-41){\circle*{1}}\put(26,-49){\circle*{1}} \put(10,-41){\circle*{1}} \put(26,-41){\circle*{1}}
\put(46,-54){$B_{1}$}
\put(36,-49){\line(1,0){24}} \put(44,-49){\line(0,1){8}} \put(36,-49){\circle*{1}} \put(44,-49){\circle*{1}} \put(52,-49){\circle*{1}} \put(60,-49){\circle*{1}} \put(44,-41){\circle*{1}}
\put(72,-54){$B_{2}$}
\put(70,-49){\line(1,0){8}} \put(70,-49){\line(0,1){8}} \put(70,-49){\line(1,1){8}} \put(78,-49){\line(0,1){8}} \put(70,-41){\line(1,0){8}} \put(70,-49){\circle*{1}} \put(78,-49){\circle*{1}} \put(70,-41){\circle*{1}} \put(78,-41){\circle*{1}}
\put(94,-54){$B_{3}$}
\put(88,-49){\circle*{1}} \put(96,-49){\circle*{1}} \put(104,-49){\circle*{1}} \put(88,-41){\circle*{1}} \put(96,-41){\circle*{1}}
\put(88,-49){\line(1,0){16}} \put(88,-49){\line(0,1){8}} \put(88,-49){\line(1,1){8}} \put(88,-41){\line(1,0){8}} \put(96,-49){\line(-1,1){8}} \put(96,-49){\line(0,1){8}} \put(104,-49){\line(-1,1){8}}
\put(135,-40){$H_{11}$}
\put(120,-40){\circle*{1}} \put(130,-40){\circle*{1}} \put(120,-50){\circle*{1}} \put(120,-30){\circle*{1}} \put(130,-30){\circle*{1}} \put(125,-25){\circle*{1}} \put(120,-30){\line(1,0){10}} \put(120,-40){\line(1,0){10}} \put(120,-30){\line(0,-1){10}} \put(130,-30){\line(0,-1){10}}
\put(125,-25){\line(-1,-1){5}} \put(125,-25){\line(-1,-3){5}} \put(125,-25){\line(1,-3){5}} \put(125,-25){\line(1,-1){5}} \put(120,-30){\line(1,-1){10}} \put(120,-40){\line(1,1){10}} \put(120,-50){\line(0,1){10}} \put(120,-50){\line(1,1){10}}
\end{picture}
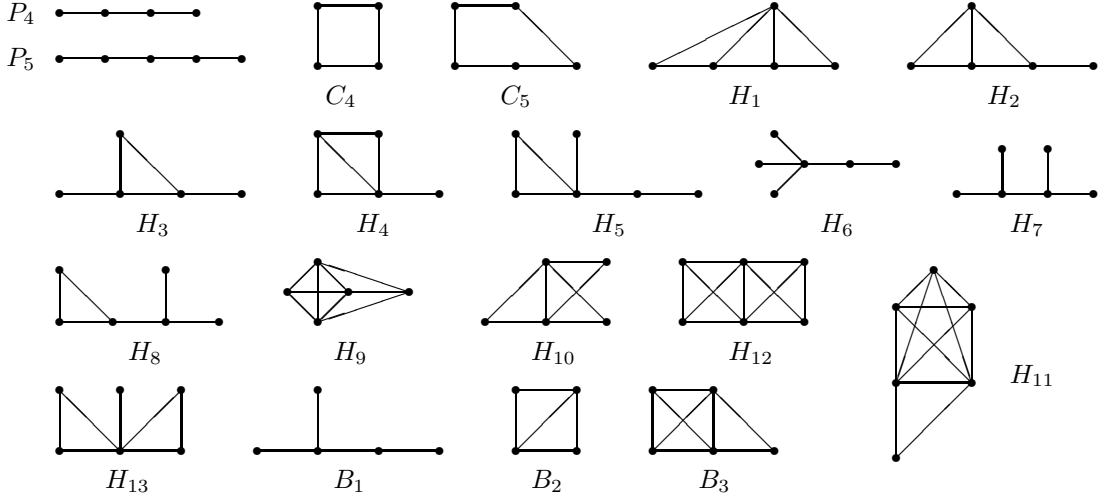
\vskip 5.5cm Fig. $3$. Graphs $P_{4}$, $P_{5}$, $C_{4}$, $C_{5},$
$H_{1}-H_{13}$ and $B_{1}-B_{3}$.
\end{center}

Next, we firstly show that $K^{h}_{n}$ is determined by its
$D$-spectrum. Let $G$ be a graph $D$-cospectral to $K^{h}_{n}$. We
call $H$ a forbidden subgraph of $G$ if $G$ contains no $H$ as an
induced subgraph.

\begin{lemma}\label{th4}
If $G$ and $K^{h}_{n}$ are $D$-cospectral, then $C_{4}$, $C_{5}$ and
$H_{i}~(i\in\{1,4,9,10,11,12,13\})$ are forbidden subgraphs of $G$.
\end{lemma}

\begin{proof}
Let $G$ and $K^{h}_{n}$ have the same distance spectrum. Suppose
that $H$ is an induced subgraph of $G$ and $H\in \{C_{4}$, $C_{5},$
$H_{i}~(i\in\{1,4,9,10,11,12,13\})\}$. Note that $\mbox{diam}(H)=2$,
obviously $D(H)$ is a principal submatrix of $D(G)$. Let $|V(H)|=m$,
by Lemma \ref{le15}, then $\lambda_{2}(G)\geq\lambda_{2}(H)$,
$\lambda_{3}(G)\geq\lambda_{3}(H)$ and
$\lambda_{m-1}(H)\geq\lambda_{n-1}(G)$. By Theorem \ref{th1}, we
know that $-1<\lambda_{2}(G)<-\frac{1}{2}$, $\lambda_{3}(G)=-1$ and
$\lambda_{n-1}(G)\in \{-1,-2\}.$ Hence we have
$\lambda_{2}(H)<-\frac{1}{2}$, $\lambda_{3}(H)\leq -1$ and
$\lambda_{m-1}(H)\geq -2$. However $\lambda_{2}\geq-\frac{1}{2}$ for
$C_{4}$, $C_{5}$ and $H_{i}~(i\in \{1,9,10,11,12\})$;
$\lambda_{3}>-1$ for $H_{4}$ and $\lambda_{m-1}<-2$ for $H_{13}$, a
contradiction. \qed
\end{proof}

\begin{center}
\setlength{\unitlength}{1.0mm}
\begin{picture}(300,20)
\put(17,0){$P_{5}$}
\put(5,9){\line(1,0){28}} \put(5,9){\circle*{1}} \put(12,9){\circle*{1}} \put(19,9){\circle*{1}} \put(26,9){\circle*{1}} \put(33,9){\circle*{1}}
\put(4,11){$v_{1}$} \put(11,11){$v_{2}$} \put(18,11){$v_{3}$} \put(25,11){$v_{4}$} \put(32,11){$v_{5}$}

\put(55,0){$H_{2}$}
\put(45,9){\line(1,0){24}} \put(45,9){\circle*{1}} \put(53,9){\circle*{1}} \put(61,9){\circle*{1}} \put(69,9){\circle*{1}}
\put(53,17){\circle*{1}} \put(53,17){\line(0,-1){8}} \put(53,17){\line(-1,-1){8}} \put(53,17){\line(1,-1){8}}
\put(44,6){$v_{1}$} \put(52,6){$v_{2}$} \put(60,6){$v_{3}$} \put(68,6){$v_{4}$} \put(52,19){$v_{5}$}

\put(93,0){$H_{3}$}
\put(83,9){\line(1,0){24}} \put(83,9){\circle*{1}} \put(91,9){\circle*{1}} \put(99,9){\circle*{1}} \put(107,9){\circle*{1}}
\put(91,17){\circle*{1}} \put(91,17){\line(0,-1){8}} \put(91,17){\line(1,-1){8}}
\put(82,6){$v_{1}$} \put(90,6){$v_{2}$} \put(98,6){$v_{3}$} \put(106,6){$v_{4}$} \put(90,19){$v_{5}$}

\put(131,0){$H_{5}$}
\put(121,9){\line(1,0){24}} \put(121,9){\circle*{1}} \put(129,9){\circle*{1}} \put(137,9){\circle*{1}} \put(145,9){\circle*{1}}
\put(121,17){\circle*{1}} \put(121,17){\line(0,-1){8}} \put(121,17){\line(1,-1){8}}
\put(129,17){\circle*{1}} \put(129,17){\line(0,-1){8}}
\put(120,6){$v_{1}$} \put(128,6){$v_{2}$} \put(136,6){$v_{3}$} \put(144,6){$v_{4}$} \put(128,19){$v_{5}$} \put(120,19){$v_{6}$}

\end{picture}
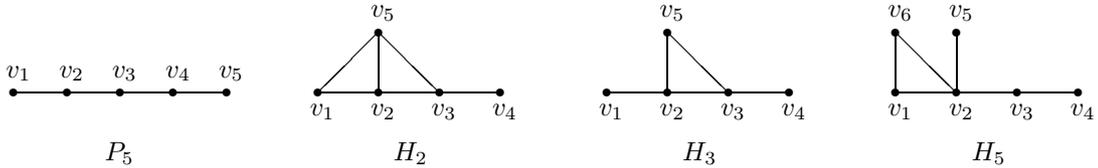
\vskip 0.5cm Fig. $4$. The labeled graphs of $P_{5}$, $H_{2}$,
$H_{3}$ and $H_{5}$.
\end{center}

\begin{tabular}{c|c|c|c|c|c|c}
\hline
 & $\lambda_{1}$ & $\lambda_{2}$ & $\lambda_{3}$ & $\lambda_{4}$ & $\lambda_{5}$ & $\lambda_{6}$\\
\hline
$P_{4}$&5.1623&-0.5858&-1.1623&-3.4142&&\\
\hline
$P_{5}$&8.2882&-0.5578&-0.7639&-1.7304&-5.2361&\\
\hline
$C_{4}$&4.0000&0.0000&-2.0000&-2.0000&&\\
\hline
$C_{5}$&6.0000&-0.3820&-0.3820&-2.6180&-2.6180&\\
\hline
$H_{1}$&5.2926&-0.3820&-0.7217&-1.5709&-2.6180&\\
\hline
$H_{2}$&6.2162&-0.4521&-1.0000&-1.1971&-3.5669&\\
\hline
$H_{3}$&6.6375&-0.5858&-0.8365&-1.8010&-3.4142&\\
\hline
$H_{4}$&5.7596&-0.5580&-0.7667&-2.0000&-2.4348&\\
\hline
$H_{5}$&9.3154&-0.5023&-1.0000&-1.0865&-2.3224&-4.4042\\
\hline
$H_{6}$&9.6702&-0.4727&-1.0566&-2.0000&-2.0000&-4.1409\\
\hline
$H_{7}$&10.0000&-0.4348&-1.0000&-2.0000&-2.0000&-4.5616\\
\hline
$H_{8}$&9.6088&-0.4931&-1.0000&-1.0924&-2.0000&-5.0233\\
\hline
$H_{9}$&4.4495&-0.4495&-1.0000&-1.0000&-2.0000&\\
\hline
$H_{10}$&5.3723&-0.3723&-1.0000&-2.0000&-2.0000&\\
\hline
$H_{11}$&6.1425&-0.4913&-1.0000&-1.0000&-1.0000&-2.6512\\
\hline
$H_{12}$&6.4641&-0.4641&-1.0000&-1.0000&-1.0000&-3.0000\\
\hline
$H_{13}$&7.8526&-0.6303&-1.0000&-1.0000&-2.2223&-3.0000\\
\hline
$B_{1}$&7.4593&-0.5120&-1.0846&-2.0000&-3.8627&\\
\hline
$B_{2}$&3.5616&-0.5616&-1.0000&-2.0000&&\\
\hline
$B_{3}$&4.9018&-0.5122&-1.0000&-1.0000&-2.3896&\\
\hline
\end{tabular}\\\\

For any $S\subseteq V(G)$, let $D_{G}(S)$ denote the principal
submatrix of $D(G)$ obtained by $S$.

\begin{lemma}\label{th13}
If $G$ and $K^{h}_{n}$ are $D$-cospectral, then $P_{5}$ and
$H_{i}~(i\in\{2,3,5,6,7,8\})$ are forbidden subgraphs of $G$.
\end{lemma}

\begin{proof}
For $P_{5}$. Suppose that $P_{5}$ is an induced subgraph of $G$,
then $d_{G}(v_{1},v_{5})\in \{2,3,4\}$. If $d_{G}(v_{1},v_{5})=4$,
then $D_{G}(\{v_{1},v_{2},v_{3},v_{4},v_{5}\})=D(P_{5})$ is a
principal submatrix of $D(G)$. By Lemma \ref{le15}, we have
$\lambda_{3}(G)\geq \lambda_{3}(P_{5})=-0.7639>-1$, a contradiction.
If $d_{G}(v_{1},v_{5})\in \{2,3\}$, let $d_{G}(v_{1},v_{4})=a$,
$d_{G}(v_{1},v_{5})=b$ and $d_{G}(v_{2},v_{5})=c$, then $a,b,c\in
\{2,3\}$. We get the principal submatrix of $D(G)$
\begin{equation*}\begin{split}
D_{G}(\{v_{1},v_{2},v_{3},v_{4},v_{5}\})&=\left(\begin{array}{cccccccc}
0&1&2&a&b\\
1&0&1&2&c\\
2&1&0&1&2\\
a&2&1&0&1\\
b&c&2&1&0\\
\end{array}\right).\\
\end{split}\end{equation*}
By a simple calculation, we have\\\\
\begin{tabular}{c|c|c|c|c|c|c|c|c}
\hline
$(a,b,c)$ & $(3,3,3)$ & $(3,2,2)$ & $(3,2,3)$ & $(3,3,2)$ & $(2,3,3)$ & $(2,3,2)$ & $(2,2,2)$ & $(2,2,3)$\\
\hline
$\lambda_{2}$&-0.4348&-0.3260&0&-0.3713&-0.3713&-0.1646&-0.2909&-0.3260\\
\hline
\end{tabular}\\\\
By Lemma \ref{le15}, $\lambda_{2}(G)\geq
\lambda_{2}(D_{G}(\{v_{1},v_{2},v_{3},v_{4},v_{5}\}))>-\frac{1}{2}.$
Note that $\lambda_{2}(G)<-\frac{1}{2},$ a contradiction. Hence
$P_{5}$ is a forbidden subgraph of $G$.

For $H_{2}$. Assume that $H_{2}$ is an induced subgraph of $G$, then
$d_{G}(v_{1},v_{4})\in \{2,3\}$. If $d_{G}(v_{1},v_{4})=3$, then
$D(H_{2})$ is a principal submatrix of $D(G)$. By Lemma \ref{le15},
we have $\lambda_{2}(G)\geq \lambda_{2}(H_{2})=-0.4521>-1/2$, a
contradiction. If $d_{G}(v_{1},v_{4})=2$, it is easy to calculate
that
$\lambda_{2}(D_{G}(\{v_{1},v_{2},v_{3},v_{4},v_{5}\}))=-0.2284>-1/2$.
By Lemma \ref{le15} and Theorem \ref{th1}, we also get a
contradiction. Therefore $H_{2}$ is a forbidden subgraph of $G$.

For $H_{3}$. Suppose that $H_{3}$ is an induced subgraph of $G$,
then $d_{G}(v_{1},v_{4})\in\{2,3\}$. If $d_{G}(v_{1},v_{4})=3$, then
$D(H_{3})$ is a principal submatrix of $D(G)$. By Lemma \ref{le15},
we have $\lambda_{3}(G)\geq \lambda_{3}(H_{3})=-0.8365>-1$, a
contradiction. If $d_{G}(v_{1},v_{4})=2$, it is easy to check that
$\lambda_{2}(D_{G}(\{v_{1},v_{2},v_{3},v_{4},v_{5}\}))=-0.3820>-1/2$.
By Lemma \ref{le15} and Theorem \ref{th1}, we also obtain a
contradiction. Hence $H_{3}$ is a forbidden subgraph of $G$.

For $H_{5}$. Assume that $H_{5}$ is an induced subgraph of $G$. If
$d_{G}(v_{1},v_{4})=d_{G}(v_{4},v_{5})=d_{G}(v_{4},v_{6})=3$, then
$D(H_{5})$ is a principal submatrix of $D(G)$. By Lemma \ref{le15},
we have $\lambda_{n-1}(G)\leq\lambda_{5}(H_{5})=-2.3224<-2$, a
contradiction. Otherwise, there exists at least one equal to 2 among
$d_{G}(v_{1},v_{4}), d_{G}(v_{4},v_{5})$ and $d_{G}(v_{4},v_{6}).$
Without loss of generality, we may assume that
$d_{G}(v_{1},v_{4})=2$. Note that $H_{5}$ is an induced subgraph of
$G$, then there exists a vertex $v\in V(G)\setminus
\{v_{1},v_{2},v_{3},v_{4},v_{5}\}$ such that $vv_{1}, vv_{4}\in
E(G).$ Then $G[vv_{1}v_{2}v_{3}v_{4}]=C_{5}$,
$G[vv_{1}v_{2}v_{3}v_{4}]=H_{1}$, $G[vv_{2}v_{3}v_{4}]=C_{4}$ or
$G[vv_{1}v_{2}v_{3}]=C_{4}$. By Lemma \ref{th4}, $C_{4}$, $C_{5}$
and $H_{1}$ are forbidden subgraphs of $G$, a contradiction. Hence
$H_{5}$ is a forbidden subgraph of $G$.

For $H_{6}$, $H_{7}$ and $H_{8}$. Suppose that they are induced
subgraphs of $G,$ respectively. If $D(H_{6})$, $D(H_{7})$ and
$D(H_{8})$ are principal submatrices of $D(G),$ respectively. By
Lemma \ref{le15}, $\lambda_{2}(G)\geq \lambda_{2}(H_{i})>-1/2$ where
$i\in\{6,7,8\}$, a contradiction. Otherwise, similar to the
discussion for $H_{5}$, we may also obtain the same contradictions.
Thus $H_{6}$, $H_{7}$ and $H_{8}$ are forbidden subgraphs of $G$. \qed
\end{proof}

\begin{theorem}\label{th5}
The graph $K^{h}_{n}$ is determined by its $D$-spectrum.
\end{theorem}

\begin{proof}
Let $G$ be a graph $D$-cospectral to $K^{h}_{n}$.
By Lemma \ref{th13}, $P_{5}$ is a forbidden graph of $G$, thus
$\mbox{diam}(G)\leq3.$ By $\lambda_{n}(G)<-2,$ then
$\mbox{diam}(G)\geq2.$

{\bf Case $1$.}\ \ $\mbox{diam}(G)=3.$

If $|V(G)|=4$, then $G=P_{4},$ it is easy to check that $G$ is not
$D$-cospectral to $K_{4}^{3}$, a contradiction. Next we assume that
$|V(G)|\geq5$. Note that $\mbox{diam}(G)=3,$ then there exists a
diameter-path $P=u\tilde{u}\tilde{v}v$ with length 3 in $G.$ Let
$X=\{u,\tilde{u},\tilde{v},v\}$, then $G[X]=P_{4}$. Denote by
$V_{i}~(i=0,1,2,3,4)$ the vertex subset of $V\backslash X$, whose
each vertex is adjacent to $i$ vertices of $X$. Clearly $V\backslash
X=\cup_{i=0}^{4}V_{i}.$

{\bf Claim $1$.} $V_{4}=\emptyset$.

Suppose not, then there exists a vertex $v_{4}\in V_{4}$ such that
$G[v_{4}u\tilde{u}\tilde{v}v]=H_{1}$, a contradiction. Hence Claim 1
holds.

{\bf Claim $2$.} $V_{3}=\emptyset$.

Suppose not, then there exists a vertex $v_{3}\in V_{3}$ such that
$v_{3}$ is adjacent to $\{u,\tilde{u},\tilde{v}\}$,
$\{\tilde{u},\tilde{v},v\}$, $\{u,\tilde{u},v\}$ or
$\{u,\tilde{v},v\}$. Then $G$ contains an induced subgraph $H_{2}$
or $C_{4}$, a contradiction.

Let $V_{2}^{u}=\{v_{2}\in V_{2} | v_{2}u, v_{2}\tilde{u}\in E(G)\}$
and $V_{2}^{v}=\{v_{2}\in V_{2} | v_{2}v, v_{2}\tilde{v}\in E(G)\}$.

{\bf Claim $3$.} $V_{2}=V_{2}^{u}\cup V^{v}_{2}$,
$G[V^{u}_{2}]~(G[V^{v}_{2}])=K_{|V^{u}_{2}|}~(K_{|V^{v}_{2}|})$ and
$E[V^{u}_{2},V^{v}_{2}]=\emptyset$.

For any $v_{2}\in V_{2},$ it is impossible that $v_{2}$ is adjacent
to $u$ and $v$ since $d_{G}(u,v)=3.$ If $v_{2}$ is adjacent to $u$
and $\tilde{v}$~(or $\tilde{u}$ and $v$), then
$G[v_{2}u\tilde{u}\tilde{v}]=C_{4}$~(or
$G[v_{2}\tilde{u}\tilde{v}v]=C_{4}$), by Lemma \ref{th4}, a
contradiction. If $v_{2}$ is adjacent to $\tilde{u}$ and
$\tilde{v}$, then $G[v_{2}u\tilde{u}\tilde{v}v]=H_{3}$, a
contradiction. Thus $V_{2}=V^{u}_{2}\cup V^{v}_{2}$. For any
$v_{2},v_{2}^{\star}\in V_{2}^{u}$, then $v_{2}v_{2}^{\star}\in
E(G).$ Otherwise $G[v_{2}v_{2}^{\star}u\tilde{u}\tilde{v}]=H_{4}$, a
contradiction. This means that $G[V^{u}_{2}]=K_{|V^{u}_{2}|}$.
Similarly, $G[V^{v}_{2}]=K_{|V^{v}_{2}|}.$ If $v_{2}v_{2}^{\star}\in
E(G)$ for any $v_{2}\in V^{u}_{2}$ and $v_{2}^{\star}\in V^{v}_{2}$,
then $G[v_{2}v_{2}^{\star}\tilde{u}\tilde{v}]=C_{4}$, a
contradiction. Hence $E[V^{u}_{2},V^{v}_{2}]=\emptyset$.

{\bf Claim $4$.} $|V_{1}|\leq 1$.

Let $v_{1}\in V_{1}$. Obviously, $v_{1}$ can only be adjacent to
$\tilde{u}$ or $\tilde{v}$, otherwise
$G[v_{1}u\tilde{u}\tilde{v}v]=P_{5}$, a contradiction. Now we assume
that $|V_{1}|\geq 2$. Let $v_{1},v_{1}^{\star}\in V_{1}$. If they
are adjacent to the same vertex of $X$, then
$G[v_{1}v_{1}^{\star}u\tilde{u}\tilde{v}v]=H_{5}$ or $H_{6}$, a
contradiction. Otherwise,
$G[v_{1}v_{1}^{\star}u\tilde{u}\tilde{v}v]=H_{7}$ or
$G[v_{1}v_{1}^{\star}\tilde{u}\tilde{v}]=C_{4}$, a contradiction.
Hence Claim 4 is completed.

{\bf Claim $5$.} Only one set is nonempty between $V_{1}$ and
$V_{2}$.

Suppose not, then there exist two vertices $v_{1}\in V_{1}$ and
$v_{2}\in V_{2}$. Without loss of generality, we may assume that
$v_{2}$ is adjacent to $u$ and $\tilde{u}$. If $v_{1}$ is adjacent
to $\tilde{u}$, then $G[v_{1}v_{2}u\tilde{u}\tilde{v}v]=H_{5}$ or
$G[v_{1}v_{2}u\tilde{u}\tilde{v}]=H_{4}$, a contradiction. If
$v_{1}$ is adjacent to $\tilde{v}$, then
$G[v_{1}v_{2}u\tilde{u}\tilde{v}v]=H_{8}$ or
$G[v_{1}v_{2}\tilde{u}\tilde{v}]=C_{4}$, a contradiction. Thus Claim
5 holds.

{\bf Claim $6$.} $V_{0}=\emptyset$.

Suppose not, then there exist a vertex $v_{0}\in V_{0} $ such that
$v_{0}v^{\star}\in E(G),$ where $v^{\star}\in V_{1}\cup V_{2}.$ Then
$G[v_{0}v^{\star}\tilde{u}\tilde{v}v]=P_{5}$ or
$G[v_{0}v^{\star}u\tilde{u}\tilde{v}]=P_{5}$, a contradiction.

By Claims 1-6, we have $V=V_{1}\cup V_{2}\cup X$. If $|V_{1}|=1$,
then by Claim 5, $V_{2}=\emptyset$. This means that $G\cong B_{1}$.
It is easy to check that $B_{1}$ has distinct $D$-spectrum with
$K_{5}^{h}$, a contradiction. So we have $V_{1}=\emptyset,$ then
$V_{2}\neq\emptyset$, and thus $G\cong K_{n}^{s+t}.$ By Corollary
\ref{co1}, $K_{n}^{s+t}$ has distinct $D$-spectrum with $K_{n}^{h},$
a contradiction. It follows that there is no graph $G$ with diameter
3 $D$-cospectral to $K_{n}^{h}.$

{\bf Case $2$.}\ \ $\mbox{diam}(G)=2.$

There exists a diameter-path $P=xyz$ with length 2 in $G.$ Let
$X=\{x,y,z\}$, then $G[X]=P_{3}$. Obviously, $V\backslash
X\neq\emptyset$ since $n\geq4$. Denote by $V_{i}~(i=0,1,2,3)$ the
vertex subset of $V\backslash X$, whose each vertex is adjacent to
$i$ vertices of $X$. Clearly $V\backslash X=\cup_{i=0}^{3}V_{i}.$

{\bf Claim $7$.} $|V_{3}|\leq 1$.

Suppose not, there exist two vertices $v_{3}, v_{3}^{\star}\in
V_{3}$. If $v_{3}v_{3}^{\star}\in E(G)$,
$G[v_{3}v_{3}^{\star}xyz]=H_{9}$, a contradiction. Otherwise
$v_{3}v_{3}^{\star}\not\in E(G)$, then
$G[v_{3}v_{3}^{\star}xz]=C_{4}$, a contradiction. Therefore Claim 7
holds.

Let $V_{xy}=\{v_{2}\in V_{2} | v_{2}x,v_{2}y\in E(G)\}$,
$V_{yz}=\{v_{2}\in V_{2} | v_{2}y,v_{2}z\in E(G)\}.$

{\bf Claim $8$.} $V_{2}=V_{xy}\cup V_{yz}$,
$G[V_{xy}]~(G[V_{yz}])=K_{|V_{xy}|}~(K_{|V_{yz}|}),$ and
$E[V_{xy},V_{yz}]=\emptyset$.

For any $v_{2}\in V_{2},$ it is impossible that $v_{2}$ is adjacent
to $x$ and $z$ since $G[v_{2}xyz]=C_{4}.$ Hence $V_{2}=V_{xy}\cup
V_{yz}.$ For any $v_{2},v_{2}^{\star}\in V_{xy}$, then
$v_{2}v_{2}^{\star}\in E(G).$ Otherwise
$G[v_{2}v_{2}^{\star}xyz]=H_{4}$, a contradiction. This means that
$G[V_{xy}]=K_{|V_{xy}|}.$ Similarly, $G[V_{yz}]=K_{|V_{yz}|}.$ If
$E[V_{xy}, V_{yz}]\neq\emptyset$, then there exist two vertices
$v_{2}\in V_{xy}$ and $v_{2}^{\star}\in V_{yz}$ such that
$v_{2}v_{2}^{\star}\in E(G)$, and thus
$G[v_{2}v_{2}^{\star}xyz]=H_{1}$, a contradiction. Hence
$E[V_{xy},V_{yz}]=\emptyset.$

{\bf Claim $9$.} If $v_{1}\in V_{1}$, then $v_{1}$ must be adjacent
to $y$.

Suppose not, then $v_{1}$ is adjacent to $x$ or $z$. Without loss of
generality, we may assume that $v_{1}x\in E(G)$. Note that
$\mbox{diam}(G)=2,$ then there exists a vertex $u\in V\backslash X $
such that $uv_{1},uz\in E(G)$, and thus $u\in \cup_{i=1}^{3}V_{i}.$
If $u\in V_{1}$, then $G[uv_{1}xyz]=C_{5}$, a contradiction. If
$u\in V_{2}$, by Claim 8, $u$ is adjacent to $y$ and $z$, and then
$G[uv_{1}xy]=C_{4}$, a contradiction. If $u\in V_{3}$, then
$G[uv_{1}xyz]=H_{1}$, a contradiction. Thus Claim 9 holds.

{\bf Claim $10$.} $V_{0}=\emptyset.$

Suppose not, then there exists a vertex $v_{0}\in V_{0}$ such that
$v_{0}$ is adjacent to some vertices of $V_{1}\cup V_{2} \cup
V_{3}.$ If $v_{0}$ is adjacent to only one vertex $u$ of $V_{1}\cup
V_{2} \cup V_{3},$ then $u\in V_{3}$ since $\mbox{diam}(G)=2,$ and
thus $G[v_{0}uxyz]=H_{4},$ a contradiction. So $v_{0}$ must be
adjacent to at least two vertices of $V_{1}\cup V_{2} \cup V_{3},$
we always find an induced subgraph $C_{4}$ of $G$ at each case, a a
contradiction. Therefore Claim 10 is obtained.

By Claim 10, $\emptyset\neq V\backslash X=\cup_{i=1}^{3}V_{i}.$ Next
we distinguish the following four cases.

{\bf Subcase $2.1$.}\ \ $V_{3}\neq\emptyset$.

By Claim 7, $|V_{3}|=1$. Note that $H_{4}$ and $H_{10}$ are
forbidden subgraphs of $G$, then $V_{1}=\emptyset$. Let
$V_{3}=\{v_{3}\}$. Obviously, $v_{2}v_{3}\in E(G)$ for each
$v_{2}\in V_{2}.$ Otherwise $G[v_{2}v_{3}xyz]=H_{1}$, a
contradiction. If $|V_{2}|\leq 2,$ i.e., there exist two vertices
$v_{2}, v_{2}^{\star}\in V_{2}$, then $G[v_{2}
v_{2}^{\star}v_{3}xyz]=H_{11}$ or $H_{12}$, a contradiction. So we
have $|V_{2}|\leq 1$. If $V_{2}=\emptyset$, then $G\cong B_{2}$, it
is easy to check that $B_{2}$ has distinct distance spectrum with
$K_{4}^{3},$ a contradiction. If $|V_{2}|=1$, then $G\cong B_{3}$.
Clearly, $B_{3}$ is not $D$-cospectral to $K^{h}_{5}$, a
contradiction.

{\bf Subcase $2.2$.}\ \ $V_{3}=\emptyset$, $V_{2}\neq\emptyset$ and
$V_{1}=\emptyset$.

By Claim 8, $G\cong K_{n}^{n-1}$ or $G\cong K^{s,t}_{n}$. By
Corollary \ref{co1}, $K^{s,t}_{n}$ and $K^{h}_{n}$ have distinct
distance spectra, a contradiction. Then $G\cong K_{n}^{n-1}$.

{\bf Subcase $2.3$.}\ \ $V_{3}=\emptyset$, $V_{2}\neq\emptyset$ and
$V_{1}\neq\emptyset$.

For any $v_{1}\in V_{1},$ we claim that $d(v_{1})=1.$ In fact, if
$d(v_{1})\geq 2,$ then there exists a vertex $v_{2}\in V_{2}$ such
that $v_{1}v_{2}\in E(G),$ and then $G[v_{1}v_{2}xyz]=H_{4},$ a
contradiction. Furthermore, we claim that only one set is nonempty
between $V_{xy}$ and $V_{yz}.$ Otherwise, let $v_{2}\in V_{xy}$ and
$v_{2}^{\star}\in V_{yz},$ then $G[v_{2}v_{2}^{\star}xyz]=H_{13},$ a
contradiction. Hence $G\cong K_{n}^{h}.$

{\bf Subcase $2.4$.}\ \ $V_{3}=\emptyset$, $V_{2}=\emptyset$ and
$V_{1}\neq\emptyset$.

Let $V_{1}^{\star}=\{v\in V_{1}|d(v)\geq2\}$. If
$V_{1}^{\star}=\emptyset$, then $G\cong K_{1,n-1}$. Note that
$\lambda_{n}(K_{1,n-1})=-2,$ then $K_{1,n-1}$ is not $D$-cospectral
to $K^{h}_{n},$ a contradiction. If $V_{1}^{\star}\neq\emptyset$, we
claim that $G[V_{1}^{\star}]=K_{|V_{1}^{\star}|}$. If not, there
exist $u,v\in V_{1}^{\star}$ such that $uv\not\in E(G)$. If there
exists a vertex $w\in V_{1}^{\star}$ such that $wu,wv\in E(G)$, then
$G[wuvxy]=H_{4}$, a contradiction. Otherwise, there exist two
distinct vertices $w_{1}\in V_{1}^{\star}$ and $w_{2}\in
V_{1}^{\star}$ such that $w_{1}u\in E(G)$ and $w_{2}v\in E(G)$, then
$w_{1}w_{2}\in E(G)$ since $H_{13}$ is a forbidden subgraph of $G$.
Thus $G[w_{1}w_{2}uvy]=H_{1}$, a contradiction. Hence
$G[V_{1}^{\star}]=K_{|V_{1}^{\star}|}$, it means that $G\cong
K_{n}^{h}.$ \qed
\end{proof}

\begin{theorem}\label{th6}
The graph $K^{s+t}_{n}$ is determined by its $D$-spectrum.
\end{theorem}

\begin{proof}
Let $G$ be a graph $D$-cospectral to $K^{s+t}_{n}$. From Theorem
\ref{th2}, we know that $-1<\lambda_{2}(G)<-\frac{1}{2}$,
$\lambda_{3}(G)=-1$ and $-2<\lambda_{n-1}(G)<-1$. Similar to the
proof of Lemmas \ref{th4} and \ref{th13}, we also get $P_{5}$,
$C_{4}$, $C_{5}$ and $H_{i}~(i=1,2,\ldots,13)$ are forbidden
subgraphs of $G$. Note that $P_{5}$ is a forbidden subgraph of $G$
and $\lambda_{n}(G)<-2$, then $2\leq\mbox{diam}(G)\leq3.$ By the
above forbidden subgraphs, similar to the proof of Theorem
\ref{th5}, we have:

$\bullet$ If $\mbox{diam}(G)=3$, then $G\cong B_{1}$ or $G\cong
K^{s+t}_{n}$.

$\bullet$ If $\mbox{diam}(G)=2$, then $G\cong B_{2},$ $G\cong
B_{3},$ $G\cong K^{h}_{n}$ or $G\cong K^{s,t}_{n}.$

From $D$-spectra of $B_{i}~(i=1,2,3)$ and Corollary \ref{co1}, then
we must have $G\cong K^{s+t}_{n}.$ Thus the theorem follows. \qed
\end{proof}

\begin{theorem}\label{th7}
The graph $K^{s,t}_{n}$ is determined by its $D$-spectrum.
\end{theorem}

\begin{proof}
Let $G$ be a graph $D$-cospectral to $K^{s,t}_{n}$. By Theorem
\ref{th3}, then $-1<\lambda_{2}(G)<-\frac{2}{3}<-\frac{1}{2}$,
$\lambda_{3}(G)=\lambda_{n-1}(G)=-1$. Hence we can still use
$P_{5}$, $C_{4}$, $C_{5}$ and $H_{i}~(i=1,2,\ldots,13)$ as the
forbidden subgraph of $G.$ Note that $P_{5}$ is a forbidden subgraph
of $G$ and $\lambda_{n}(G)<-2$, then $2\leq\mbox{diam}(G)\leq3.$
Similar to the proof of Theorem \ref{th5}, then

$\bullet$ If $\mbox{diam}(G)=3$, then $G\cong B_{1}$ or $G\cong
K^{s+t}_{n}$.

$\bullet$ If $\mbox{diam}(G)=2$, then $G\cong B_{2},$ $G\cong
B_{3},$ $G\cong K^{h}_{n}$ or $G\cong K^{s,t}_{n}.$

By $D$-spectra of $B_{i}~(i=1,2,3)$ and Corollary \ref{co1}, then
$G\cong K^{s,t}_{n}.$ Thus $K^{s,t}_{n}$ is determined by its
$D$-spectrum. \qed
\end{proof}

In \cite{LRF}, Liu et al. give the
distance characteristic polynomial of $K^{n_{1},n_{2},\ldots,
n_{k}}_{n}$:
$$P_{D}(\lambda)=(\lambda+1)^{n-k-1}(\lambda-\sum_{i=1}^{k}\frac{n_{i}(2\lambda+1)}{\lambda+n_{i}+1})\prod_{i=1}^{k}(\lambda+n_{i}+1).$$

Next, we will show that $K^{n_{1},n_{2},\ldots, n_{k}}_{n}~(1\leq n_{i}\leq 2)$ is
determined by its $D$-spectrum.

\begin{center}
\setlength{\unitlength}{1.0mm}
\begin{picture}(112,14)

\put(10,0){\line(1,0){24}}\put(44,0){\line(1,0){24}}\put(78,0){\line(1,0){24}}
\put(22,0){\line(0,1){12}}\put(56,0){\line(0,1){12}}\put(90,0){\line(0,1){12}}
\put(44,0){\line(2,1){12}}\put(78,0){\line(2,1){12}}\put(102,0){\line(-2,1){12}}
\put(10,0){\circle*{1}}\put(22,0){\circle*{1}}\put(34,0){\circle*{1}}\put(22,6){\circle*{1}}
\put(22,12){\circle*{1}}\put(44,0){\circle*{1}}\put(56,0){\circle*{1}}\put(68,0){\circle*{1}}
\put(78,0){\circle*{1}}\put(90,0){\circle*{1}}\put(102,0){\circle*{1}}
\put(56,6){\circle*{1}}\put(90,6){\circle*{1}}\put(56,12){\circle*{1}}\put(90,12){\circle*{1}}
\put(9,-4){$x$}\put(43,-4){$x$}\put(77,-4){$x$}
\put(21,-4){$y$}\put(55,-4){$y$}\put(89,-4){$y$}
\put(33,-4){$z$}\put(67,-4){$z$}\put(101,-4){$z$}
\put(23,12){$v$}\put(57,12){$v$}\put(91,12){$v$}
\put(23,6){$w$}\put(57,6){$w$}\put(91,6){$w$}
\put(20,-12){$T_{1}$}\put(54,-12){$T_{2}$}\put(88,-12){$T_{3}$}
\end{picture}
\vskip 1.5cm Fig. $5$. Graphs $T_{1}-T_{3}$.
\end{center}

\begin{theorem}\label{th8}
$K^{n_{1},n_{2},\ldots, n_{k}}_{n}~(1\leq n_{i}\leq 2)$ is determined by its D-spectrum.
\end{theorem}

\begin{proof}
Let $G:=K^{n_{1},n_{2},\ldots, n_{k}}_{n},$ where $1\leq n_{i}\leq 2.$ Let $t_{1}$ and $t_{2}$ be two nonnegative integers with $t_{1}+t_{2}=k$. Suppose that  $n_{1}=\cdots=n_{t_{1}}=1$ and $n_{t_{1}+1}=\cdots=n_{t_{1}+t_{2}}=2$. Clearly, if $t_{1}=0,$ then $G$ is the friendship graph $F_{n}^{k}$. If $t_{2}=0,$ then $G$ is a star. Recall that the star is determined by its $D$-spectrum. So we assume that $t_{2}\geq 1$. Note that the distance characteristic polynomial of $G$ is
$$P_{D}(\lambda)=(\lambda+1)^{n-t_{1}-t_{2}-1}(\lambda+2)^{t_{1}-1}(\lambda+3)^{t_{2}-1}(\lambda^{3}+(5-4t_{2}-2t_{1})\lambda^{2}+(6-10t_{2}-7t_{1})\lambda-3t_{1}-4t_{2}).$$
Consider the cubic function
$$f(\lambda)=\lambda^{3}+(5-4t_{2}-2t_{1})\lambda^{2}+(6-10t_{2}-7t_{1})\lambda-3t_{1}-4t_{2}.$$
By calculation, we have
\[
\left\{
\begin{array}{l}
f(0)=-3t_{1}-4t_{2}<0,\\
f(-\frac{1}{2})=-\frac{15}{8},\\
f(-1)=2t_{1}+2t_{2}-2\geq 0,\\
f(-2)=3t_{1}\geq 0,\\
f(-3)=-10t_{2}<0.
\end{array}
\right.
\]
Then the three roots of $f(\lambda)=0$ belong to the intervals $(0, +\infty)$, $[-1,-\frac{1}{2})$ and $(-3,-2]$, respectively. Consequently, we have $-1\leq \lambda_{2}(G)<-\frac{1}{2}$, $\lambda_{3}(G)=-1$ and $\lambda_{n}(G)=-3$.

Suppose that $G'$ is $D$-cospectral to $G$, that is $-1\leq \lambda_{2}(G')<-\frac{1}{2}$, $\lambda_{3}(G')=-1$ and $\lambda_{n}(G')=-3$. In the following, we only need show that $G'\cong G$. It is easy to see that $G'$ can not contain $P_{4}$ as an induced subgraph, otherwise we have $\lambda_{n}(G')\leq\lambda_{4}(P_{4})=-3.4142$, which contradicts $\lambda_{n}(G')=-3$. Thus the diameter of $G'$ is 2. Let $P=xyz$ be a diameter path of $G'$.

\noindent{\bf Claim $1$.} $d_{G'}(y)=n-1.$

If there exists a vertex $v\in V(G')$ such that $vy\notin E(G'),$ then
$d_{G'}(v,y)=2$, and thus
\begin{equation*}\begin{split}
D_{G'}(\{x,y,z,v\})&=\left(\begin{array}{cccccccc}
0&1&2&a\\
1&0&1&2\\
2&1&0&b\\
a&2&b&0\\
\end{array}\right).\\
\end{split}\end{equation*}
Then $a,b\in\{1,2\}$, by a simple calculation, we have
\setlength{\tabcolsep}{29pt} 
\renewcommand\arraystretch{1.1}  
\begin{center}
\begin{tabular}{ccccc}
\hline
$(a,b)$ & $(1,1)$ & $(1,2)$ & $(2,1)$ & $(2,2)$\\
\hline
$\lambda_{2}$&0.0000&-0.3820&-0.3820&-0.6519\\
\hline
\end{tabular}
\end{center}
By Lemma \ref{le15}, only the case $a=2, b=2$ satisfies $\lambda_{2}(G')<-\frac{1}{2}$. Thus there exists a vertex $w$
such that the subgraph of $G'$ induced by vertices $v,w,x,y,z$ is
$T_{1}$, $T_{2}$ or $T_{3}$ (see Fig. 5). We get a principal
submatrix of $D(G')$ for each case.
\begin{equation*}\begin{split}
D_{1}=\left(\begin{array}{ccccccc}
0&1&2&2&2\\
1&0&1&1&2\\
2&1&0&2&2\\
2&1&2&0&1\\
2&2&2&1&0\\
\end{array}\right),
D_{2}=\left(\begin{array}{ccccccc}
0&1&2&1&2\\
1&0&1&1&2\\
2&1&0&2&2\\
1&1&2&0&1\\
2&2&2&1&0\\
\end{array}\right),
D_{3}=\left(\begin{array}{ccccccc}
0&1&2&1&2\\
1&0&1&1&2\\
2&1&0&1&2\\
1&1&1&0&1\\
2&2&2&1&0\\
\end{array}\right).
\end{split}\end{equation*}\\
A simple calculation gives $\lambda_{2}(D_{1})=-0.2248$,
$\lambda_{2}(D_{2})=-0.3820$ and $\lambda_{3}(D_{3})=-0.7667$. For
each case, Cauchy interlacing theorem implies $\lambda_{2}(G')\geq\lambda_{2}(D_{1})=-0.2248$,
$\lambda_{2}(G')\geq\lambda_{2}(D_{2})=-0.3820$ and
$\lambda_{3}(G')\geq\lambda_{3}(D_{3})=-0.7667$, a contradiction.
Thus Claim 1 holds.

\noindent{\bf Claim $2$.} $G'-y$ is the disjoint union of some cliques.

According to Lemma \ref{le18}, we obtain $G'$ has $n-1+t_{2}$ edges. It follows from Claim 1 that $G'-y$ has $t_{2}$ edges. Since $t_{2}\leq \lfloor\frac{n-1}{2}\rfloor$, there are at least two connected components in $G'-y$. Suppose that there is a component which is not a clique. Then we can see that $H_{4}$ is an induced subgraph of $G'$. Therefore $\lambda_{3}(G')\geq \lambda_{3}(H_{4})=-0.7667$, a contradiction. Thus Claim 2 holds.

Combining Claims 1 and 2, we have $G'\cong K_{1}\vee(K_{n_{1}^{'}}\cup K_{n_{2}^{'}}\cup\cdots\cup K_{n_{t}^{'}})$. According to the distance characteristic polynomial of $G$ and $G'$, we have $t=k$ and $n_{i}^{'}=n_{i},$ i.e. $G'\cong G$, as desired.\hspace*{\fill}$\Box$
\end{proof}

The following result follows from Theorem \ref{th8} immediately.

\begin{corollary}{\bf (\cite{LL})}\label{co10}
The friendship graph $F_{n}^{k}$ is determined by its $D$-spectrum.
\end{corollary}

\noindent{\bf Acknowledgment}

The authors would like to thank the anonymous referees very much for
valuable suggestions and corrections which improve the original
manuscript.

\small {

}
\end{document}